\documentclass[a4paper,11pt]{amsart}
\usepackage{amsfonts,amsthm,amsmath,amssymb,graphicx,bbold}
\theoremstyle{plain}
\newtheorem{lem}{Lemma}[section]
\newtheorem{prop}[lem]{Proposition}
\newtheorem{thm}[lem]{Theorem}
\newtheorem{cor}[lem]{Corollary}
\theoremstyle{definition}

\theoremstyle{remark}

\DeclareMathOperator{\cls}{cls}
\DeclareMathOperator{\rank}{rank}
\DeclareMathOperator{\disc}{disc}

\DeclareMathOperator{\sym}{sym}

\DeclareMathOperator{\e}{e}

\DeclareMathOperator{\diag}{diag}

\newcommand{\bmu}{\boldsymbol \mu}
\newcommand{\bnu}{\boldsymbol \nu}
\newcommand{\bdelta}{\boldsymbol \delta}
\newcommand{\bbeta}{\boldsymbol \beta}
\newcommand{\bgamma}{\boldsymbol \gamma}

\newcommand{\calP}{{\mathfrak P}}

\newcommand{\kfld}{{\mathbb K}}

\newcommand{\Z}{\mathbb Z}
\newcommand{\Q}{\mathbb Q}

\newcommand{\Hyp}{\mathbb H}
\newcommand{\F}{\mathbb F}

\newcommand{\G}{\mathcal G}

\newcommand{\W}{\mathcal W}

\newcommand{\Ok}{\mathcal O}

\begin{document}

\title[Twisted Gauss sums and totally isotropic subspaces]
{Higher level quadratically twisted Gauss sums and totally isotropic subspaces}

\author{Lynne Walling}
\address{School of Mathematics, University of Bristol, University Walk, Clifton, Bristol BS8 1TW, United Kingdom;
phone +44 (0)117 331-5245}
\email{l.walling@bristol.ac.uk}


\begin{abstract}  We consider a generalized Gauss sum supported on matrices over a number field.  We evaluate this Gauss sum and relate it to the number of totally isotropic subspaces of related quadratic spaces.  Then we consider a further generalization of such a Gauss sum, realizing its value in terms of numbers of totally isotropic subspaces of related quadratic spaces.
\end{abstract}

\maketitle
\def\thefootnote{}
\def\thefootnote{\arabic{footnote}}

\section{Introduction}
\smallskip

Gauss sums and their numerous generalizations are ubiquitous in number theory.
When studying the action of Hecke operators on half-integral weight Hilbert-Siegel modular forms, the generalized Gauss sum we encounter is defined as follows.
Let $\kfld$ be a number field with $\Ok$ its ring of integers, $\calP$ a nondyadic prime ideal in $\Ok$, and $\F=\Ok/\calP$; we fix $\rho\in\partial^{-1}\calP^{-1}$ so that $\rho\Ok_{\calP}=\partial^{-1}\calP^{-1}\Ok_{\calP}$ (where $\partial$ is the different of $\kfld$).
Then
for $T\in\F^{n,n}_{\sym}$ (meaning that $T$ is a symmetric $n\times n$ matrix over $\F$), we set
$$\G^*_T(\calP)=\sum_{S\in\F^{n,n}_{\sym}}
\left(\frac{\det S}{\calP}\right)\e\{2TS\rho\}$$
where $\sigma$ denotes the matrix trace map,
$\e\{*\}=\exp(\pi i Tr^{\kfld}_{\Q}(\sigma(*)))$,
and $\left(\frac{*}{\calP}\right)$ is the Legendre symbol.
One sees that for $M,N\in\Ok^{n,n}_{\sym}$ with $M\equiv N\ (\text{mod }\calP)$, we have
$\e\{2M\rho\}=\e\{2N\rho\}$; consequently, $\G^*_T(\calP)$ is well-defined, although it is dependent on the choice of $\rho$.  

For our application to half-integral weight Hecke operators, we need to relate these Gauss sums to 
$R^*(T\perp\big<1\big>,0_a)$, which is 
the number of $a$-dimensional totally isotropic subspaces of the dimension $n+1$ $\F$-space $V$ whose quadratic form is given by $T\perp\big<1\big>$.  
(A subspace $W$ of $V$ is totally isotropic if the quadratic form restricted to $W$ is 0, and $A\perp B$ denotes the block-diagonal matrix $\diag(A,B)$.)  
In Theorem 1.1 we evaluate $\G^*_T(\calP)$, and in Corollary 1.2 we give
$\G^*_T(\calP)$ in terms of $R^*(T\perp\big<1\big>,0_a)$.

To state the theorem,  set
$\varepsilon=\left(\frac{-1}{\calP}\right),$ and fix $\omega\in\F$ so that $\omega$ is not a square in $\F$; set $J_n=I_{n-1}\perp\big<\omega\big>$.  For $T,S\in\F^{n,n}_{\sym}$,  write $T\sim S$ if there is some $G\in GL_n(\F)$ so that $T=\,^tGSG$.  Note that with $d=\rank T$,  either $T\sim I_d\perp 0_{n-d}$ or $T\sim J_d\perp 0_{n-d}$. 
With this notation, we have the following.

\begin{thm}  Take $T\in\F^{n,n}_{\sym}$ where $n\in\Z_+$.
Suppose that $0\le d\le n$ and
$T\sim I_d\perp 0_{n-d}$ or $T\sim J_d\perp 0_{n-d}$.  Take $c$ so that $d=2c$ or $d=2c+1$.
\begin{enumerate}
\item[(a)]  Suppose that $n=2m$.  Then with $N(\calP)$ the norm of $\calP$,
$$\G^*_T(\calP)=(-1)^c\varepsilon^m N(\calP)^{m^2}\cdot
\prod_{i=1}^{m-c}(N(\calP)^{2i-1}-1).$$
\item[(b)]  Suppose that $n=2m+1$. When $d=2c$, $\G^*_T(\calP)=0$.  When $d=2c+1$,
$$\G^*_T(\calP)=(-1)^c\varepsilon^{m+c}N(\calP)^{m^2+2m-c}\,\G^*_1(\calP)\cdot
\prod_{i=1}^{m-c}(N(\calP)^{2i-1}-1)$$
if $T\sim I_d\perp 0_{n-d}$, and 
$\G^*_T(\calP)=-\G^*_{I_d\perp 0_{n-d}}(\calP)$
if $T\sim J_d\perp 0_{n-d}$.  
\end{enumerate}
\end{thm}

\begin{cor}  Take $T\in\F^{n,n}_{\sym}$ where $n\in\Z_+$.
Let $V$ be the dimension $n+1$ space over $\F$ with quadratic form given by $T\perp\big<1\big>$, and
let $R^*(T\perp\big<1\big>,0_a)$ be the number of $a$-dimensional  totally isotropic subspaces of $V$.
We have
$$\left(\G_1^*(\calP)\right)^n\G^*_T(\calP)
=\sum_{a=0}^n (-1)^{n+a}N(\calP)^{n(n+1)/2+a(a-n)}
R^*(T\perp\big<1\big>,0_a).$$
\end{cor}

To prove Theorem 1.1, we perform a deconstruction to  reduce $\G^*_T(\calP)$ to a sum in terms of Gauss sums $\G^*_Y(\calP)$ where the $Y$ are smaller than  $T$.  For this we repeatedly use the elementary fact that $\sigma(AB)=\sigma(BA)$, knowledge of representation numbers over finite fields, and elementary combinatorial methods.  Then using induction, we prove Theorem 1.1; Corollary 1.2 then follows  from Lemma 3.1 of \cite{half-int-aps}.

In Proposition 4.1, we consider the following generalized Gauss sum:
with $T\in\F^{n,n}_{\sym}$ and $0\le r\le n$, set
 $$\G^*_T(\calP;r)=
\sum_{S\sim I_r\perp 0_{n-r}}\e\{2TS\rho\}-\sum_{S\sim J_r\perp 0_{n-r}}\e\{2TS\rho\}$$
(so for $r=n$, this is $\G^*_T(\calP)$).
We again deconstruct $\G^*_T(\calP;r)$ as a sum in terms of Gauss sums $\G^*_Y(\calP)$ with the $Y$ smaller than $T$, and then from this and Theorem 1.1, we describe $\G^*_T(\calP;r)$ in terms of numbers of totally isotropic subspaces of spaces of quadratic spaces related to $T$.

It is important for us to note that in \cite{S}, Saito studies analogues of these Gauss sums over finite fields, with an interest to applications to twists of Siegel modular forms.  
Although his main interest is in twists by the quadratic and the trivial characters, he considers twists by all characters, making his arguments more complicated than ours.  We note that Theorem 1.3 \cite{S} includes the results of our Theorem 1.1.
Saito also considers finite field analogues of the  Gauss sums
$\G^*_T(\calP;r)$.
He develops relations between these Gauss sums, some of which are quite complicated.  In Proposition 4.1 (a), we present a simple relation very similar to his relation in Proposition 1.12 \cite{S}; then in Proposition 4.1 (b) we present formulas for these Gauss sums in terms of numbers of totally isotropic subspaces.
The value of this paper is to present an approach simpler than that of \cite{S}, demonstrating our deconstruction technique,
and to relate these Gauss sums to representations of zeros.

Note that it is quite easy to modify our techniques  to
 generalized Gauss sums twisted by the trivial character, and to
Gauss sums over
a finite field $\F_q$ with odd characteristic $p$ where
  $\e\{*\}$ is replaced by  $\exp(\pi i Tr^{\F_q}_{\F_p}(\sigma(*))/p)).$

\bigskip
\section{Notation}
\smallskip

Besides the notation given in the introduction, we define the following.

For $t,s\in \Z_+$ with $s\le t$, and $X\in\F^{t,t}_{\sym}$, $Y\in\F^{s,s}_{\sym}$,  define the
representation number $r(X,Y)$ to be
$$r(X,Y)=\#\{C\in\F^{t,s}:\ ^tCXC=Y\ \},$$
and define the primitive representation number $r^*(X,Y)$ to be
$$r^*(X,Y)=\#\{C\in\F^{t,s}:\ ^tCXC=Y,\ \rank C=s\ \}.$$
Let $o(X)$ denote the order of the orthogonal group of $X$; so $o(X)=r^*(X,X).$
We make great use of the following elementary functions, that help us encode formulas involving
representation numbers.
\begin{align*}
\bmu(t,s)&=\prod_{i=0}^{s-1}(N(\calP)^{t-i}-1),\ 
\bdelta(t,s)=\prod_{i=0}^{s-1}(N(\calP)^{t-i}+1),\\
\bbeta(t,s)&=\frac{\bmu(t,s)}{\bmu(s,s)},\ 
\bnu(t,s)=\prod_{i=s}^{t-1}(N(\calP)^t-N(\calP)^i),\ 
\bgamma(t,s)=\frac{\bmu\bdelta(t,s)}{\bmu\bdelta(s,s)}.
\end{align*}
We agree that when $s=0$, the value of any of these functions is 1; when $s<0$, we agree that $\bbeta(t,s)=0$.  Note that $\bbeta(t,s)$ is the number of $s$-dimensional subspaces of a $t$-dimensional space over $\F$, and $\bnu(t,0)$ is the number of bases for a $t$-dimensional space.
Finally, for $d\in\Z_+$ and $i\in\Z$ with $0\le i\le d$, we set
$U_{d,i}=I_i\perp 0_{d-i}$ and $\overline U_{d,i}=J_i\perp 0_{d-i}.$

\bigskip
\section{Proofs of Theorem 1.1 and Corollary 1.2}
\smallskip

We begin by proving Theorem 1.1.
As $\calP$ is fixed, in this section we write $\G^*_T$ for $\G^*_T(\calP)$.

First notice that
$$\G^*_{0_n}=\sum_{Y\sim I_n}1-\sum_{Y\sim J_n}1
=\frac{|GL_n(\F)|}{o(I_n)}-\frac{|GL_n(\F)|}{o(J_n)};$$
so using Lemma 5.1, when $n$ is odd we get
$\G^*_{0_n}=0$, and when $n=2m$ we get $$\G^*_{0_n}=\varepsilon^mN(\calP)^{m^2}\frac{\bmu(2m,2m)}{\bmu\bdelta(m,m)}.$$

For the rest of this section, take $d$ so that $0<d<n$.
With $G\in GL_n(\F)$, we have 
$$\e\{2\,^tGI_nGU_{n,d}\cdot\rho\}
=\e\{2Y'\rho\},\ 
\e\{2\,^tGI_nG\overline U_{n,d}\cdot\rho\}
=\e\{2Y'J_d\rho\}$$ 
where $Y'$
 is the upper left block of $^tGI_nG$; similarly,
$$\e\{2\,^tGJ_nGU_{n,d}\cdot\rho\}
=\e\{2Y'\rho\},\ 
\e\{2\,^tGJ_nG\overline U_{n,d}\cdot\rho\}
=\e\{2Y'J_d\rho\}$$
 where $Y'$
is the upper left block of $^tGJ_nG$.
The number of $Y\sim I_n$ with upper left $d\times d$ block $Y'$ is
$\bnu(n,d) r^*(I_n,Y')/o(I_n),$
as for $C\in\F^{n,d}$ with $\rank C=d$, the number of ways to extend $C$ to an element of $GL_n(\F)$ is $\bnu(n,d)$.
Similarly, the number of $Y\sim J_n$ with upper left $d\times d$ block $Y'$ is $\bnu(n,d) r^*(J_n,Y')/o(J_n).$
Hence we have
\begin{align*}
\G^*_{U_{n,d}}
&=\sum_{Y\sim I_n} \e\{2YU_{n,d}\cdot\rho\}
-\sum_{Y\sim J_n} \e\{2YU_{n,d}\cdot\rho\}\\
&=\sum_{G\in GL_n(\F)}
\left( \frac{\e\{2\,^tGI_nGU_{n,d}\cdot\rho\}}{o(I_n)}
- \frac{\e\{2\,^tGJ_nGU_{n,d}\cdot\rho\}}{o(J_n)}\right)\\
&=\bnu(n,d)\sum_{Y'\in\F^{d,d}_{\sym}}
\left(\frac{r^*(I_n,Y')}{o(I_n)}-\frac{r^*(J_n,Y')}{o(J_n)}\right)\e\{2Y'\rho\}.
\end{align*}
Note that we can partition $\F^{d,d}_{\sym}$ into $GL_d(\F)$-orbits, and
in Lemma 5.1, we compute representation numbers $r^*(\cdot,\cdot)$.
We find that when $n$ is odd, we have $o(I_n)=o(J_n)$,
$r^*(I_n, U_{d,2k})-r^*(J_n,U_{d,2k})=0,$
and
$$r^*(I_n,U_{d,2k+1})-r^*(J_n,U_{d,2k+1})
=r^*(J_n,\overline U_{d,2k+1})-r^*(I_n,\overline U_{d,2k+1}).$$
Hence with $n=2m+1$, using Lemma 5.1 and then Lemma 5.3, we get
\begin{align*}
\G^*_{U_{n,d}}
&=\frac{\bnu(n,d)}{o(I_n)}
\sum_{k=0}^{d/2} 2\varepsilon^{m-k}N(\calP)^{2mk-k^2+m-k+(d-2k-1)(d-2k-2)/2}\\
&\quad\cdot
\bmu\bdelta(m,d-k-1)
\left(\sum_{Y\sim U_{d,2k+1}}\e\{2Y\rho\}
-\sum_{Y\sim\overline U_{d,2k+1}}\e\{2Y\rho\}\right)\\
&\quad=\frac{\bnu(n,d)}{o(I_n)}
\sum_{k=0}^{d/2} 2\varepsilon^{m-k}N(\calP)^{2mk-k^2+m-k+(d-2k-1)(d-2k-2)/2}\\
&\quad\cdot
\bmu\bdelta(m,d-k-1)
\cdot\sum_{\cls Y\in\F^{2k+1,2k+1}_{\sym}}\frac{r^*(I_d,Y)}{o(Y)}\, \G^*_Y
\end{align*}
(where $\cls Y$ is the isometry class of $Y$, or equivalently, the $GL_d(\F)$-orbit of $Y$).
With $n=2m+1$, similar reasoning gives us
\begin{align*}
\G^*_{\overline U_{n,d}}
&=\bnu(n,d)\sum_{Y'\in\F^{d,d}_{\sym}}
\left(\frac{r^*(I_n,Y')}{o(I_n)}-\frac{r^*(J_n,Y')}{o(J_n)}\right)
\e\{2Y'J_d\rho\}\\
&\quad=\frac{\bnu(n,d)}{o(I_n)}
\sum_{k=0}^{d/2} 2\varepsilon^{m-k}N(\calP)^{2mk-k^2+m-k+(d-2k-1)
(d-2k-2)/2}\\
&\quad\cdot
\bmu\bdelta(m,d-k-1)\cdot
\sum_{\cls Y\in\F^{2k+1,2k+1}_{\sym}}\frac{r^*(J_d,Y)}{o(Y)}\, \G^*_Y.
\end{align*}

Now suppose that $n=2m$.  Then using Lemma 5.1 we have
\begin{align*}
&\frac{r^*(I_n,Y)}{o(I_n)}-\frac{r^*(J_n,Y)}{o(J_n)}\\
&\quad=
\frac{1}{o(I_{n+1})}
(r^*(I_{2m+1},\big<1\big>\perp Y)-r^*(J_{2m+1},\big<1\big>\perp Y)).
\end{align*}
So following the above reasoning, for $n=2m$ we have
\begin{align*}
\G^*_{U_{n,d}}
&=\frac{\bnu(n,d)}{o(I_{n+1})}
\sum_{k=0}^{d/2} 2\varepsilon^{m-k}N(\calP)^{2mk-k^2+m-k+(d-2k)(d-2k-1)/2}\\
&\quad\cdot
\bmu\bdelta(m,d-k)\cdot
\sum_{\cls Y\in\F^{2k,2k}_{\sym}}\frac{r^*(I_d,Y)}{o(Y)}\, \G^*_Y,\\
\G^*_{\overline U_{n,d}}
&=\frac{\bnu(n,d)}{o(I_{n+1})}
\sum_{k=0}^{d/2} 2\varepsilon^{m-k}N(\calP)^{2mk-k^2+m-k+(d-2k)(d-2k-1)/2}\\
&\quad\cdot
\bmu\bdelta(m,d-k)\cdot
\sum_{\cls Y\in\F^{2k,2k}_{\sym}}\frac{r^*(J_d,Y)}{o(Y)}\, \G^*_Y.
\end{align*}

To evaluate $\G^*_{I_n}$ and $\G^*_{J_n}$, we make use of the (non-twisted) Gauss sums
\begin{align*}
\G_{I_n}&=\sum_{U\in\F^{n,n}}\e\{2I_n[U]\rho\},\ 
\G_{J_n}=\sum_{U\in\F^{n,n}}\e\{2J_n[U]\rho\},\\
\overline\G_{I_n}
&=\sum_{U\in\F^{n,n}}\e\{2I_n[U]J_n\rho\},\ 
\overline\G_{J_n}
=\sum_{U\in\F^{n,n}}\e\{2J_n[U]J_n\rho\}.
\end{align*}
For $Y\in\F^{n,n}$, by looking at the trace of the matrix $^tYY$, it is easy to check that 
 $\G_{I_n}=(\G_1^*)^{n^2}$.
Similarly, we have 
\begin{align*}
\G_{J_n}&=(\G_1^*)^{n(n-1)}\cdot(\G^*_{\omega})^{n}=
\overline\G_{I_n},\ 
\overline\G_{J_n}=(\G^*_1)^{(n-1)^2}\cdot(\G^*_{\omega})^{2n-1}.
\end{align*}
Classical techniques give us $\G^*_{\omega}=-\G^*_1$
and $(\G^*_1)^2=\varepsilon N(\calP)$.

On the other hand, we have
\begin{align*}
\G_{I_{n}}&=\sum_{Y\in\F^{n,n}_{\sym}} r(I_{n}, Y) \e\{2Y\rho\},\ 
\G_{J_{n}}=\sum_{Y\in\F^{n,n}_{\sym}} r(J_{n}, Y) \e\{2Y\rho\},\\ 
\overline\G_{I_{n}}&=\sum_{Y\in\F^{n,n}_{\sym}} r(I_{n}, Y) \e\{2YJ_n\rho\},\ 
\overline\G_{J_{n}}=\sum_{Y\in\F^{n,n}_{\sym}} r(J_{n}, Y) \e\{2YJ_n\rho\}.
\end{align*}
Partitioning $\F^{n,n}_{\sym}$ into $GL_{n}(\F)$-orbits,
we get
\begin{align*}
&\frac{1}{o(I_n)}\,\G_{I_n}-\frac{1}{o(J_n)}\,\G_{J_n}\\
&\quad=
\frac{r(I_n,0_n)}{o(I_n)}-\frac{r(J_n,0_n)}{o(J_n)}\\
&\qquad+
\sum_{0<\ell\le n}\sum_{Y\sim U_{n,\ell}} \left(\frac{r(I_n,U_{n,\ell})}{o(I_n)}-
\frac{r(J_n,U_{n,\ell})}{o(J_n)}\right) \e\{2Y\rho\}\\
&\qquad +
\sum_{0<\ell\le n}\sum_{Y\sim\overline U_{n,\ell}} \left(\frac{r(I_n,\overline U_{n,\ell})}{o(I_n)}-
\frac{r(J_n,\overline U_{n,\ell})}{o(J_n)}\right) \e\{2Y\rho\}.
\end{align*}
Notice that $r(I_n,I_{\ell}\perp 0_{d-\ell})=r^*(I_n,I_{\ell})r(I_{n-\ell},0_{d-\ell}).$
So using Lemmas 5.1 and 5.2, and then Lemma 5.3, when $n$ is odd we get
\begin{align*}
\G_{I_n}-\G_{J_n}
&\quad=
\sum_{\substack{0\le \ell\le n\\ \ell\,\text{odd}}}
(r(I_{n}, U_{n,\ell})-r(J_{n}, U_{n,\ell}))\\
&\qquad\cdot
\left(\sum_{Y\sim U_{n,\ell}}\e\{2Y\rho\}
-\sum_{Y\sim \overline U_{n,\ell}}\e\{2Y\rho\}\right)\\
&\quad=
\sum_{\substack{0\le \ell\le n\\ \ell\,\text{odd}}}
(r(I_{n}, U_{n,\ell})-r(J_{n}, U_{n,\ell}))
\sum_{\cls Y\in\F^{\ell,\ell}_{\sym}}
\frac{r^*(I_{n},Y)}{o(Y)}\, \G^*_Y.
\end{align*}
Similarly, with $n$ odd,
\begin{align*}
\overline\G_{I_{2m+1}}-\overline\G_{J_{2m+1}}
&=\sum_{\substack{0\le \ell\le n\\ \ell\,\text{odd}}}
(r(I_{n}, U_{n,\ell})-r(J_{n}, U_{n,\ell}))
\sum_{\cls Y\in\F^{\ell,\ell}_{\sym}}
\frac{r^*(J_{n},Y)}{o(Y)}\, \G^*_Y.
\end{align*}
 When $n$ is even, similar arguments give us
\begin{align*}
&r^*(I_{n+1},1)\,\G_{I_{n}}-r^*(J_{n+1},1)\,\G_{J_{n}}\\
&=
\sum_{\substack{0\le \ell\le n\\ \ell\,\text{even}}}
(r(I_{n+1}, \big<1\big>\perp U_{n,\ell})-r(J_{n+1}, \big<1\big>\perp U_{n,\ell}))
\sum_{\cls Y\in\F^{\ell,\ell}_{\sym}}
\frac{r^*(I_{n},Y)}{o(Y)}\, \G^*_Y,
\end{align*}
\begin{align*}
&r^*(I_{n+1},1)\,\overline\G_{I_{n}}-r^*(J_{n+1},1)\,\overline\G_{J_{n}}\\
&=
\sum_{\substack{0\le \ell\le n\\ \ell\,\text{even}}}
(r(I_{n+1}, \big<1\big>\perp U_{n,\ell})-r(J_{n+1}, \big<1\big>\perp U_{n,\ell}))
\sum_{\cls Y\in\F^{\ell,\ell}_{\sym}}
\frac{r^*(J_{n},Y)}{o(Y)}\, \G^*_Y.
\end{align*}

Now  we argue by induction on $m$ to prove the theorem in the case that $n=2m+1$.  For $m=0$, we have $\G^*_{U_{1,1}}=\G^*_1$ (by definition of $\G^*_1$), and as we have already noted, $\G^*_{\overline U_{1,1}}=-\G^*_1$.
So suppose that $m\ge1$ and that the theorem holds for all $\G^*_Y$ where
$Y\in\F^{2r+1,2r+1}_{\sym}$ and $0\le r<m$.
With $0<d<n$, we begin with the expression for $\G^*_{U_{n,d}}$ that we derived above.
By the induction hypothesis, for
$2k+1\le d$ and
$Y\in\F^{2k+1,2k+1}_{\sym}$, we have
$\G^*_Y=\varepsilon^k N(\calP)^{k^2+2k}\cdot h_Y$
where $h_Y$ is defined in Lemma 5.4.
So by Lemma 5.4 we have
$\G^*_{U_{n,d}}=0$ when $d$ is even, and
\begin{align*}
\G^*_{U_{n,d}}
&=\frac{\bnu(n,d)}{o(I_n)}
\sum_{k=0}^{d/2} 2\varepsilon^{m-k}N(\calP)^{2mk-k^2+m-k+(d-2k-1)(d-2k-2)/2}\\
&\quad\cdot
\bmu\bdelta(m,d-k-1)\cdot (-1)^k\varepsilon^{k+c} N(\calP)^{k^2+c}\,\bgamma(c,k)
\end{align*}
when $d$ is odd with $d=2c+1$.

So assume now that $d=2c+1$; then we have
\begin{align*}
 \G^*_{U_{n,d}}
&= \frac{\bnu(n,d)}{o(I_n)}
\,2\varepsilon^{m+c} N(\calP)^{m+2c^2}\G^*_1\cdot A(c,0)
\text{  where }\\
A(t,q)&=\sum_{k=0}^t(-1)^k N(\calP)^{2k(m+k-2t-q)}\bmu\bdelta(m,2t+q-k)
\bgamma(t,k).
\end{align*}
Since $\gamma(t,k)=N(\calP)^{2k}\bgamma(t-1,k)+\bgamma(t-1,k-1)$,
we have
\begin{align*}
A(t,q)
&=\sum_{k=0}^{t-1} (-1)^k N(\calP)^{2k(m+k+1-2t-q)}
\bmu\bdelta(m,2t+q-k-1))\bgamma(t-1,k)\\
&\quad\cdot
(\bmu\bdelta(m,2t+q-k)-N(\calP)^{2(m-2t-q+k+1)})\\
&=-A(t-1,q+1)=(-1)^tA(0,t)=(-1)^t \bmu\bdelta(m,t+q).
\end{align*}
Therefore, using that 
$\bnu(n,d)=N(\calP)^{(n-d)(n+d-1)/2}\bmu(n-d,n-d)$, 
\begin{align*}
\G^*_{U_{n,d}}
&=(-1)^c \varepsilon^{m+c} N(\calP)^{m^2+2m-c}\cdot
\frac{\bmu(2(m-c),2(m-c))}{\bmu\bdelta(m-c,m-c)}\, \G^*_1,
\end{align*}
as claimed in the statement of the theorem.
A virtually identical argument gives us $\G^*_{\overline U_{n,d}}=-\G^*_{U_{n,d}}$.

Now, still taking $n=2m+1$ and beginning with our earlier expression for $\G_{I_n}-\G_{J_n},$
we use Lemmas 5.1 and 5.2 to give us
\begin{align*}
\G_{I_n}-\G_{J_n}&=2\sum_{k=0}^{m}\sum_{s=0}^{m-k}(-1)^{m-k-s}
\varepsilon^{m-k}N(\calP)^{2mk-k^2+m-k+(m-k)^2+s^2}\\
&\quad\cdot
\bmu\bdelta(m,k)
\bbeta\bdelta(m-k,s)
\cdot\sum_{\cls Y\in\F^{2k+1,2k+1}_{\sym}}
\frac{r^*(I_{2m+1},Y)}{o(Y)}\, \G^*_Y.
\end{align*}
Using that $o(I_n)=2N(\calP)^{m^2}\bmu\bdelta(m,m)$,
and the induction hypothesis for $Y\in\F^{2k+1,2k+1}_{\sym}$ with $k<m$, we have \begin{align*}
\G_{I_n}-\G_{J_n}&=
o(I_{2m+1})\G^*_{I_{2m+1}}
-o(I_{2m+1})\varepsilon^mN(\calP)^{m^2+2m}\G^*_1\, h_{I_{2m+1}}+2\G^*_1B
\end{align*}
where 
\begin{align*}B&=\sum_{k=0}^m\sum_{s=0}^{m-k}(-1)^{m-k-s}
\varepsilon^{m-k}N(\calP)^{m^2+m-k+s^2}\bmu\bdelta(m,k)\bbeta\bdelta(m-k,s)\\
&\quad\cdot
\varepsilon^kN(\calP)^{k^2+2k}
\sum_{\cls Y\in\F^{2k+1,2k+1}_{\sym}}
\frac{r^*(I_{2m+1},Y)}{o(Y)}\, h_Y.
\end{align*}
By Lemma 5.4, we get
\begin{align*}
B&=\sum_{k=0}^m\sum_{s=0}^{m-k}(-1)^{m-s}N(\calP)^{m^2+2m+k(k-1)+s^2}
\bmu\bdelta(m,k)\bbeta\bdelta(m-k,s)\bgamma(m,k).
\end{align*}
Since
$\bgamma(m,k)\bbeta\bdelta(m-k,s)
=\bbeta\bdelta(m,s)\bgamma(m-s,k),$
we can sum on $0\le s\le m$, $0\le k\le m-s$.  Then replacing $s$ by $m-s$ and using that $\bbeta(m,m-s)=\bbeta(m,s),$ we get
\begin{align*}
&B=\sum_{s=0}^m(-1)^sN(\calP)^{2m^2+2m-2ms+s^2}
\bdelta(m,m-s)\bbeta(m,s)\cdot C(s)
\text{  where}\\
&C(t)=\sum_{k=0}^t N(\calP)^{k(k-1)}\bmu\bdelta(m,k)\bgamma(t,k).
\end{align*}
Since $\bgamma(t,k)=N(\calP)^{2k}\bgamma(t-1,k)+\bgamma(t-1,k-1),$ we have
\begin{align*}
C(t)
&=N(\calP)^{2m}C(t-1)=N(\calP)^{2mt}C(0)=N(\calP)^{2mt}.
\end{align*}
Therefore
\begin{align*}
&B=N(\calP)^{2m^2+2m}D(m,0)
\text{  where}\\
&D(t,q)=\sum_{s=0}^t (-1)^s N(\calP)^{s(s+q)}\bdelta(t+q,t-s)\bbeta(t,s).
\end{align*}
Since $\bbeta(t,s)=N(\calP)^s\bbeta(t-1,s)+\bbeta(t-1,s-1),$ we have
\begin{align*}
D(t,q)
&=D(t-1,q+1)=D(0,t+q)=1.
\end{align*}
Therefore $B=N(\calP)^{2m^2+2m}$.  Our earlier computations show that 
$\G_{I_n}-\G_{J_n}=2 N(\calP)^{2m^2+2m}\G^*_1$;
so
$$\G^*_{I_{2m+1}}=\varepsilon^mN(\calP)^{m^2+2m}\cdot h_{I_{2m+1}}\G^*_1
=(-1)^mN(\calP)^{m^2+m}\,\G^*_1.$$
A virtually identical argument gives us $\G^*_{J_n}=-\G^*_{I_n}$.

Now we argue by induction on $m$ to prove the theorem in the case that $n=2m$.
Since the computation for $m=1$ is essentially identical to the induction step for $m>1$, we formally define $\G^*_{I_0}=\G^*_{J_0}=1$ (which is consistent with the formula claimed in the theorem).  So now suppose that $m\ge1$ and that the theorem holds for all $\G^*_Y$ where $Y\in\F^{2r,2r}_{\sym}$ and $0\le r<m$.
With $0<d<n$, we begin with the expression for $\G^*_{U_{n,d}}$ that we derived above.
Take $c$ so that $d=2c$ or $d=2c+1$.
Using the induction hypothesis, Lemma 5.4, and arguing as we did when $n$ was odd, we get
$$\G^*_{U_{n,d}}=\frac{\bnu(n,d)}{o(I_{n+1})}
2\varepsilon^m N(\calP)^{m+d(d-1)/2}A(c,d-2c)$$
where $A(t,q)$ is as defined earlier in this proof; recall that
$$A(t,q)=(-1)^t\bmu\bdelta(m,t+q).$$
Since $\bmu(2(m-c),1)=\bmu\bdelta(m-c,1)$, for $d=2c$ or $2c+1$
we get
$$\G^*_{U_{n,d}}
=(-1)^c\varepsilon^mN(\calP)^{m^2}
\frac{\bmu(2(m-c),2(m-c))}{\bmu\bdelta(m-c,m-c)}.$$
A virtually identical argument gives us $\G^*_{\overline U_{n,d}}=\G^*_{U_{n,d}}.$

Still assuming that $n=2m$ and beginning with our earlier expression for 
$r^*(I_{n+1},1)\G_{I_n}-r^*(J_{n+1},1)\G_{J_n}$, we use Lemmas 5.1 and 5.2 and the induction hypothesis to get
\begin{align*}
&r^*(I_{n+1},1)\G_{I_n}-r^*(J_{n+1},1)\G_{J_n}\\
&\quad=o(I_{n+1})\G^*_{I_n}-o(I_{n+1})\varepsilon^m N(\calP)^{m^2}
\cdot h_{I_n}
+2\varepsilon^m N(\calP)^{-m}B
\end{align*}
where $B$ is as in the case of $n$ odd.  We saw that $B=N(\calP)^{2m^2+2m}$, and 
$$r^*(I_{n+1},1)\G_{I_n}-r^*(J_{n+1},1)\G_{J_n}
=2\varepsilon^m N(\calP)^{2m+m};$$
so we get 
$$\G^*_{I_{2m}}=\varepsilon^m N(\calP)^{m^2}\cdot h_{I_{2m}}
=(-1)^m\varepsilon^m N(\calP)^{m^2}.$$
The argument to evaluate $\G^*_{J_{2m}}$ is essentially identical to that of evaluating $\G^*_{J_{2m+1}}$, where for this we begin with the identity
$$r^*(J_{2m+1},1)\overline\G_{I_{2m}}-r^*(I_{2m+1},1)\overline\G_{J_{2m}}
=2\varepsilon^m N(\calP)^{2m^2+m}.$$
This proves the theorem.

To prove Corollary 1.2, we first note that by Theorem 1.1, $\left(\G^*_1\right)^m\G^*_T$ has no dependence on our choice of $\rho$.  Thus we can follow the argument of Lemma 3.1 \cite{half-int-aps}, as the techniques are local.  In \cite{half-int-aps}, all quadratic forms were assumed to be even; since $2$ is a unit in $\F$, we have $R^*(T\perp\big<1\big>,0_a)=
R^*(2T\perp\big<2\big>,0_a)$, and hence Corollary 1.2 follows.

\bigskip
\section{Variations on quadratically twisted Gauss sums}
\smallskip

For $T\in\F^{n,n}_{\sym}$ and $0\le r\le n$, here we consider 
$\G_T^*(\calP;r)$, as defined in the introduction.
For $T=0_n$, we have $\G^*_{0_n}(\calP;r)=\G^*_{0_d}(\calP)$, so we only need to consider $T\not=0_n$.

\begin{prop}  Take $n\in\Z_+$, $T\in\F^{n,n}_{\sym}$ and let $d=\rank T$.  
\begin{enumerate}
\item[(a)]  Suppose that $0\le 2t+1\le n$.  When $d$ is even we have
$\G^*_T(\calP;2t+1)=0.$  When $d$ is odd with $d=2c+1$, we have
$\G^*_{\overline U_{n,d}}(\calP;2t+1)=-\G^*_{U_{n,d}}(\calP;2t+1),$ and
\begin{align*}
\G^*_{U_{n,d}}(\calP;2t+1)
&=\frac{\bnu(n,2c+1)}{\bnu(n-1,2c)}\varepsilon^cN(\calP)^c\G^*_1(\calP) 
\G^*_{U_{n-1,2c}}(\calP;2t).
\end{align*}

\item[(b)]  Suppose that $0\le 2t\le n$; set $s=n-2t$.  Then with $c$ so that $d=2c$ or $2c+1$,
we have 
\begin{align*}
\G^*_{T}(\calP;2t)
&=
\frac{\bnu(n,d)}{o(I_{2t+1}\perp0_s)}
\sum_{k=0}^c
(-1)^k\varepsilon^kN(\calP)^{s(2k+1)+2tk+t-k}\bmu\bdelta(t,k) \\
&\quad\cdot
\bgamma(c,k) A_s(t-k,c-k)
\end{align*}
where
\begin{align*}
A_s(x,y)=
(N(\calP)^x+\varepsilon^x) r^*(I_{2x}\perp0_s,0_{2y})
 - (N(\calP)^x-\varepsilon^x)r^*(J_{2x}\perp0_s,0_{2y}).
\end{align*}

\end{enumerate}
\end{prop}

\noindent{\bf Remark:}  As $\bnu(2y,0)$ is the number of bases for any $2y$-dimensional space,
$r^*(T',0_{2y})=\bnu(2y,0)\cdot R^*(T',0_{2y})$ for any symmetric $T'$.

\begin{proof}
Throughout this proof, we follow the lines of argument used in Section 3. In this way we get
\begin{align*}
\G^*_{U_{n,d}}(\calP;r)
&=\bnu(n,d)\sum_{Y\in\F^{d,d}_{\sym}}
\left(\frac{r^*(U_{n,r},Y)}{o(U_{n,r},Y)}
-\frac{r^*(\overline U_{n,r})}{o(\overline U_{n,r},Y)}\right) \e\{2Y/p\},
\end{align*}
\begin{align*}
\G^*_{\overline U_{n,d}}(\calP;r)
&=\bnu(n,d)\sum_{Y\in\F^{d,d}_{\sym}}
\left(\frac{r^*(U_{n,r},Y)}{o(U_{n,r},Y)}
-\frac{r^*(\overline U_{n,r})}{o(\overline U_{n,r},Y)}\right) \e\{2YJ_d/p\}.
\end{align*}
We have $o(0_{n-r})=\bnu(n-r,n-r),$ and
$$o(U_{n,r})=o(I_r)N(\calP)^{r(n-r)}o(0_{n-r}),\  
o(\overline U_{n,r})=o(J_r)N(\calP)^{r(n-r)}\bnu(n-r,n-r).$$

First consider the case that $r=2t+1$.  Then for even $\ell$ ($\ell\le d$), we have
$$r^*(U_{n,r},U_{d,\ell})-r^*(\overline U_{n,r},U_{d,\ell})
=0=r^*(U_{n,r}\overline U_{d,\ell})-r^*(\overline U_{n,r},\overline U_{d,\ell}),$$
and for odd $\ell$ we have
$$r^*(U_{n,r},U_{d,\ell})=r^*(\overline U_{n,r},\overline U_{d,\ell}),\ 
r^*(U_{n,r},\overline U_{d,\ell})=r^*(\overline U_{n,r},U_{d,\ell}).$$
Hence using Lemma 5.3, we have
\begin{align*}
\G^*_{U_{n,d}}(\calP;2t+1)
&=\frac{\bnu(n,d)}{o(U_{n,2t+1})}
\sum_{0\le 2k+1\le d}
\left(\sum_{\cls Y\in\F^{2k+1,2k+1}_{\sym}}\frac{r^*(I_d,Y)}{o(Y)} \G^*_Y\right)\\
&\quad\cdot
(r^*(U_{n,2t+1},U_{d,2k+1})-r^*(\overline U_{n,2t+1},U_{d,2k+1})).
\end{align*}
So by Theorem 1.1 and Lemmas 5.1 and 5.4, when $d$ is even we get 
$\G^*_{U_{n,d}}(\calP;2t+1)=0$, and when $d$ is odd with $d=2c+1$, we get
\begin{align*}
\G^*_{U_{n,d}}(\calP;2t+1)
&=\frac{\bnu(n,d)}{o(U_{n,2t+1})}
\sum_{k=0}^c (-1)^k\varepsilon^{k+c} N(\calP)^{k^2+c} \bgamma(c,k) \G^*_1\\
&\quad\cdot 
(r^*(U_{n,2t+1},U_{d,2k+1})-r^*(\overline U_{n,2t+1},U_{d,2k+1})).
\end{align*}
An almost identical argument gives us 
$\G^*_{\overline U_{n,d}}(\calP;2t+1)=-\G^*_{U_{n,d}}(\calP;2t+1).$

Now consider the case that $r=2t$.
With $d=2c$ or $2c+1$, reasoning as in Section 3 gives us
\begin{align*}
\G^*_{U_{n,d}}(\calP;2t) 
&=\frac{\bnu(n,d)}{o(U_{n+1,2t+1})}
\sum_{k=0}^c(-1)^k\varepsilon^k N(\calP)^{k^2} \bgamma(c,k)\\
&\quad\cdot\left(r^*(U_{n+1,2t+1},U_{d+1,2k+1})
 - r^*(\overline U_{n+1,2t+1},U_{d+1,2k+1})\right)\\
&=
\G^*_{\overline U_{n,d}}(\calP;2t).
\end{align*}
This gives us (a) and part of (b).

To finish proving (b), we begin with the above equation, taking $d=2c$.  
It is easily seen that
$r^*(I_r\perp 0_s,1)=N(\calP)^s r^*(I_r,1),$
and consequently from Lemma 5.1 we get
\begin{align*}
&
r^*(U_{n+1,2t+1},U_{d+1,2k+1})
 - r^*(\overline U_{n+1,2t+1},U_{d+1,2k+1})\\
&\quad =
N(\calP)^{(n-2t)(2k+1)+2tk-k^2+t-k}\bmu\bdelta(t,k) A_{n-2t}(t-k,c-k)
\end{align*}
with $A_s(x,y)$ as in the statement of the proposition.
\end{proof}

\bigskip
\section{Lemmas and their proofs}
\smallskip

\begin{lem}  Take $t,d\in\Z_+$. We have:
\begin{align*}
r^*(I_{2t},1)&=N(\calP)^{t-1}(N(\calP)^t-\varepsilon^t)=r^*(I_{2t},\omega),\\
r^*(J_{2t},1)&=N(\calP)^{t-1}(N(\calP)^t+\varepsilon^t)=r^*(J_{2t},1),\\
r^*(I_{2t+1},1)&=N(\calP)^t(N(\calP)^t+\varepsilon^t)=r^*(J_{2t+1},\omega),\\
r^*(I_{2t+1},\omega)&=N(\calP)^t(N(\calP)^t-\varepsilon^t)=r^*(J_{2t+1},1);
\end{align*}
also,
\begin{align*}
r^*(I_{2t},0_d)&=N(\calP)^{d(d-1)/2}(N(\calP)^t-\varepsilon^t)\bmu\bdelta(t-1,d-1)
(N(\calP)^{t-d}+\varepsilon^t),\\
r^*(J_{2t},0_d)&=N(\calP)^{d(d-1)/2}(N(\calP)^t+\varepsilon^t)\bmu\bdelta(t-1,d-1)
(N(\calP)^{t-d}-\varepsilon^t),\\
r^*(I_{2t+1},0_d)&=N(\calP)^{d(d-1)/2}\bmu\bdelta(t,d)=r^*(J_{2t+1},0_d).
\end{align*}
\end{lem}

\begin{proof}
The first collection of formulas 
are from Theorems 2.59 and 2.60 of \cite{Ger}.
For the second collection of formulas, 
we begin with  Theorems 2.59 and 2.60 of \cite{Ger}, giving us formulas for $r(I_t,0)=r^*(I_t,0)+1$ and $r(J_t,0)=r^*(J_t,0)+1$.
Now consider the case that $V$ is a $2t$-dimensional space over $\F$ equipped with a quadratic form $Q_V$ given by $I_{2t}$ relative to some basis for $V$. 
So $r^*(I_{2t},0_d)$ is the number of all (ordered) bases for $d$-dimensional, totally isotropic subspaces of $V$.  
(Recall that a subspace $W$ of $V$ is totally isotropic if $Q_V$ restricts to 0 on $W$.)
Suppose that $d>1$; we construct all bases for $d$-dimensional, totally isotropic subspaces of $V$ as follows.
Choose an isotropic vector $x$ from $V$ (so $x\not=0$ and $Q_V(x)=0$; note that this is not possible if $t=1$ and $\varepsilon=-1$).  Then as $V$ is a regular space, there is some $y\in V$ so that $y$ is not orthogonal to $x$; hence (by Theorem 2.23 \cite{Ger})
 $x,y$ span a hyperbolic plane, and (by Theorem 2.17 \cite{Ger}), this hyperbolic plane splits $V$, giving us
$V=(\F x\oplus\F y)\perp V'$ where $V'$ is hyperbolic if and only if $V$ is.
We have $\disc V=\varepsilon \disc V'$ and so the quadratic form on $V'$ is given by $I_{2(t-1)}$ if $\varepsilon=1$, and by
$J_{2(t-1)}$ if $\varepsilon=-1$.
The number of all bases for $d$-dimensional, totally isotropic subspaces of $V$ with $x$ as the first basis element is
$N(\calP)^{d-1} r^*(I_{2(t-1)},0_{d-1})$ if $\varepsilon=1$, and
$N(\calP)^{d-1} r^*(J_{2(t-1)},0_{d-1})$ otherwise.  The formula claimed now follows by induction on $d$.

Virtually identical arguments yield the formulas when $I_{2t}$ is replaced by $J_{2t}$ or $I_{2t+1}$ or $J_{2t+1}$.
\end{proof}

\begin{lem}  
Suppose that $m\ge 0$.  We have
\begin{align*}
&\sum_{s=0}^m (-1)^s N(\calP)^{(2m+1)(m-s)+s^2}\bbeta\bdelta(m,m-s)\\
&\quad=r(I_{2m+1},0_{2m+1})
=r(J_{2m+1},0_{2m+1}),
\end{align*}
\begin{align*}
&\sum_{s=0}^m(-1)^s N(\calP)^{2m(m-s)+s(s-1)}\bbeta(m,m-s)\bdelta(m-1,m-s)\\
&\quad=\begin{cases} r(I_{2m},0_{2m})&\text{if $\varepsilon^m=1$,}\\
r(J_{2m},0_{2m})&\text{if $\varepsilon^m=-1$,}\end{cases}\\
&\sum_{s=1}^m (-1)^{s+1} N(\calP)^{2m(m-s)+s(s-1)}\bbeta(m-1,m-s)\bdelta(m,m-s)\\
&\quad=\begin{cases} r(J_{2m},0_{2m})&\text{if $\varepsilon^m=1$,}\\
r(I_{2m},0_{2m})&\text{if $\varepsilon^m=-1$.}\end{cases}
\end{align*}
\end{lem}

\begin{proof}
Suppose that $V$ is an $n$-dimensional vector space over $\F$ equipped with a quadratic form given by $Q_V=I_n$ or $J_n$.  
Then $r(Q_V,0_n)$ is the number of (ordered) $x_1,\ldots,x_n\in V$ so that $\text{span}\{x_1,\ldots,x_n\}$ is totally isotropic.
As $\bnu(d,0)$ is the number of bases for any given dimension $d$ space over $\F$, the number of dimension $d$ totally isotropic subspaces of $V$ is
$$\varphi_d(V)=r^*(Q_V,0_d)/\bnu(d,0).$$

We treat the case that $V\simeq\Hyp^m$, meaning that 
$\dim V=2m$ and the quadratic form on $V$ is given by 
$I_{2m}$ if $\varepsilon^m=1$, and by $J_{2m}$ otherwise
(analogous arguments treat the other cases).  Slightly abusing notation, we write $(x_1,\ldots,x_{2m})\subseteq V$ to mean that $(x_1,\ldots,x_{2m})$ is an ordered $2m$-tuple of vectors from $V$.  We set
$$\W_{m-s}=\{\text{dimension } m-s \text{ totally isotropic subspaces } W \text{ of }V\},$$
and we let
$${\mathbb 1}_W(x_1,\ldots,x_{2m})=
\begin{cases}1&\text{if $x_1,\ldots,x_{2m}\in W$,}\\
0&\text{otherwise}.\end{cases}$$
Thus for $(x_1,\ldots,x_{2m})\subseteq V$,
$\sum_{W\in\W_{m-s}} {\mathbb 1}_W(x_1,\ldots,x_{2m})$ is the number of elements of $\W_{m-s}$ containing $x_1,\ldots,x_{2m}$, and, noting that 
$N(\calP)^{2m(m-s)}$ is the number of (ordered) $2m$-tuples of vectors in each $W\in\W_{m-s}$, we have
$$\sum_{(x_1,\ldots,x_{2m})\subseteq V} \left(\sum_{W\in\W_{m-s}}
{\mathbb 1}_W(x_1,\ldots,x_{2m})\right)
 = N(\calP)^{2m(m-s)}\varphi_{m-s}(V).$$
So
\begin{align*}
\psi(V):=
&\sum_{s=0}^m (-1)^s N(\calP)^{s(s-1)+2m(m-s)}\varphi_{m-s}(V)\\
=&\sum_{(x_1,\ldots,x_{2m})\subseteq V}\left(\sum_{s=0}^m (-1)^s N(\calP)^{s(s-1)}
\sum_{W\in\W_{m-s}}
{\mathbb 1}_W(x_1,\ldots,x_{2m})\right).
\end{align*}

Fix $(x'_1,\ldots,x'_{2m})\subseteq V$; let $W'$ be the subspace spanned by $x'_1,\ldots,x'_{2m}$, and set $\ell=\dim W'$.  If $W'$ is not totally isotropic then 
${\mathbb 1}_W(x'_1,\ldots,x'_{2m})=0$ for all totally isotropic $W$.
So suppose that $W'$ is totally isotropic.  Then repeatedly using Theorems 2.19, 2.23, 2.52 of \cite{Ger} and the assumption that $V$ is regular, we find that there is a dimension $\ell$ subspace $W''$ so that $W'\oplus W''\simeq\Hyp^{\ell}$ and
$V=(W'\oplus W'')\perp V'$ where  $V'\simeq \Hyp^{m-\ell}$.
Hence the number of $W\in\W_{m-s}$ that contain $W'$ is $\varphi_{m-s-\ell}(V')$.
Therefore, using Lemma 5.1 and the above formula for $\varphi_{m-s-\ell}(V')$, we have
\begin{align*}
&\sum_{s=0}^{m-\ell} (-1)^s N(\calP)^{s(s-1)} \sum_{W\in\W_{m-s}}
{\mathbb 1}(x'_1,\ldots,x'_{2m})
=A(m-\ell,m-\ell-1)
\end{align*}
where
$$A(t,k)=\sum_{s=0}^t (-1)^sN(\calP)^{s(s+k-t)}\bdelta(k,t-s)\bbeta(t,t-s).$$
We argue by induction on $t$ to show that for any $k$ and $t\ge 0$,
we have $A(t,k)=1$.
Clearly $A(0,k)=1$ for all $k$.  
So fix $t\ge0$ and suppose that $A(t,k)=1$ for all $k$.
Hence we have
\begin{align*}
1
&=(N(\calP)^k+1)\sum_{s=0}^t(-1)^sN(\calP)^{s(s+k-1-t)}
\bdelta(k-1,t-s)\bbeta(t,t-s)\\
&\quad -N(\calP)^k\sum_{s=0}^t (-1)^s N(\calP)^{s(s+k-t)}\bdelta(k,t-s)\bbeta(t,t-s).
\end{align*}
Notice that in the first sum in the above equality, we can allow $s$ to vary from $0$ to $t+1$ (since $\bbeta(t,-1)=0$), and in the second sum we can allow $s$ to vary from $-1$ to $t$ (since $\bbeta(t,t+1)=0$).
Also, we know that 
$$(N(\calP)^k+1)\bdelta(k-1,t-s)=\bdelta(k,t-s+1);$$
so replacing $s$ by $s-1$ in the second sum and using that
$$\bbeta(t,t-s)+N(\calP)^{t+1-s}\bbeta(t,t+1-s)=\bbeta(t+1,t+1-s),$$
we find that $A(t+1,k)=1$ for all $k$.  
Hence  $\psi$ counts  $(x_1,\ldots,x_{2m})\subseteq V$
 zero times if
$\text{span}\{x_1,\ldots,x_{2m}\}$ is not totally isotropic, and once otherwise.  Thus $\psi=r(Q_V,0_{2m}).$
\end{proof}

\begin{lem}  
Fix $d\in\Z_+$ and $\ell\in\Z$ so that $0\le \ell<d$.
Then
$$\sum_{Y\sim U_{d,\ell}} \e\{2Y\rho\}
-\sum_{Y\sim \overline U_{d,\ell}} \e\{2Y\rho\}
=\sum_{\cls Y'\in\F^{\ell,\ell}_{\sym}} \frac{r^*(I_d,Y')}{o(Y')} \G^*_{Y'}(\calP I_{\ell})$$
and
$$\sum_{Y\sim U_{d,\ell}} \e\{2YJ_d\rho\}
-\sum_{Y\sim \overline U_{d,\ell}} \e\{2YJ_d\rho\}
=\sum_{\cls Y'\in\F^{\ell,\ell}_{\sym}} \frac{r^*(J_d,Y')}{o(Y')} \G^*_{Y'}(\calP I_{\ell}),$$
where $\cls Y'$ varies over a set of representatives for the $GL_{\ell}(\F)$-orbits in $\F^{\ell,\ell}_{\sym}$.
\end{lem}

\begin{proof}
We first consider the sum over $Y\sim \overline U_{d,\ell}$.  
We know that for $G\in GL_{d}(\F)$ and $G'$ in the orthogonal group of 
$\overline U_{d,\ell}$, we have 
$^t(G'G)\overline U_{d,\ell}(G'G)=\,^tG\overline U_{d,\ell}G$, so when we let $G$ vary over $GL_{d}(\F)$, each element in the orbit of $\overline U_{d,\ell}$ appears exactly $o(\overline U_{d,\ell})$ times.  Also,
recall that with $\sigma$ denoting the matrix trace map, we have 
$\sigma(^tG\overline U_{d,\ell}G)=\sigma(\overline U_{d,\ell}GI_d\,^tG)$
and $\sigma(\overline U_{d,\ell}GI_d\,^tG)=\sigma(J_{\ell}Y')$ where $Y'$ is the upper left $\ell\times \ell$ block of $GI_d\,^tG$.
So  we have
\begin{align*}
\sum_{Y\sim\overline U_{d,\ell}}\e\{2Y\rho\}
&=\frac{1}{o(\overline U_{d,\ell})} 
\sum_{G\in GL_d(\F)} \e\{2\,^tG\overline U_{d,\ell}G\rho\}\\
&=\frac{\bnu(d,\ell)}{o(\overline U_{d,\ell})}
\sum_{Y'\in\F^{\ell,\ell}_{\sym}} r^*(I_d,Y') \e\{2 J_{\ell} Y'\rho\}
\end{align*}
since
$$r^*(I_d,Y')=\#\{C\in\F^{d,\ell}: \ ^tCC=Y',\ \rank C=\ell\ \},$$
and the number of ways to extend $C$ to an element of $GL_{d}(\F)$ is $\bnu(d,\ell)$.
Now, as $G$ varies over $GL_{\ell}(\F)$, $^tGY'G$ varies $o(Y')$ times over the elements in $\cls Y'$.  Also, by Lemma 5.1, we have
$o(\overline U_{d,\ell})=o(J_{\ell})\bnu(d,\ell)$.
Hence
\begin{align*}
\sum_{Y\sim\overline U_{d,\ell}} \e\{2Y\rho\}
&= \frac{1}{o(J_{\ell})} 
\sum_{G\in GL_d(\F)} \sum_{\cls Y'\in\F^{\ell,\ell}_{\sym}}
\frac{r^*(I_d,Y')}{o(Y')}\e\{2G J_{\ell}\,^tGY'\rho\}\\
&= \sum_{\cls Y'\in\F^{\ell,\ell}_{\sym}}\frac{r^*(I_d,Y')}{o(Y')}
\sum_{X\sim J_{\ell}}\e\{2XY'\rho\}
\end{align*}
where for the last equality we used that as $G$ varies over $GL_{\ell}(\F)$,
$GJ_{\ell}\,^tG$ varies $o(J_{\ell})$ times over the elements in the orbit of $J_{\ell}$.

The analysis of 
$$\sum_{Y\sim U_{d,\ell}}\e\{2Y\rho\},\ 
\sum_{Y\sim U_{d,\ell}}\e\{2YJ_d\rho\},\text{ and }
\sum_{Y\sim \overline U_{d,\ell}}\e\{2YJ_d\rho\}$$
follow in a virtually identical manner.  Then we note that
$$\sum_{X\sim I_{\ell}}\e\{2XY'\rho\}-\sum_{X\sim J_{\ell}}\e\{2XY'\rho\}
=\G^*_{Y'},$$
completing the proof.
\end{proof}

\begin{lem}  Suppose that $0< \ell\le d$; take $c$ so that $d$ is $2c$ or $2c+1$.
Take $Y\in\F^{\ell,\ell}_{\sym}$, and take $b$ so that $\rank Y$ is $2b$ or $2b+1$.
\begin{enumerate}
\item[(a)]  Suppose that $\ell=2k$; set
$$h_Y=(-1)^b\cdot\frac{\bmu(2(k-b),2(k-b))}{\bmu\bdelta(k-b,k-b)}.$$
Then
$$\sum_{\cls Y\in\F^{\ell,\ell}_{\sym}}\frac{r^*(I_d,Y)}{o(Y)}\,h_Y=
\sum_{\cls Y\in\F^{\ell,\ell}_{\sym}}\frac{r^*(J_d,Y)}{o(Y)}\,h_Y=
(-1)^k\bgamma(c,k).$$
\item[(b)]  Suppose that $\ell=2k+1$; set
$$h_Y=\begin{cases}
(-1)^b\varepsilon^b N(\calP)^{-b}\cdot\frac{\bmu(2(k-b),2(k-b))}{\bmu\bdelta(k-b,k-b)}
&\text{if $Y\sim I_{2b+1}\perp 0_{2(k-b)}$,}\\
(-1)^{b+1}\varepsilon^b N(\calP)^{-b}\cdot\frac{\bmu(2(k-b),2(k-b))}{\bmu\bdelta(k-b,k-b)}
&\text{if $Y\sim J_{2b+1}\perp 0_{2(k-b)}$,}\\
0&\text{if $\rank Y=2b$.}
\end{cases}$$
Then when $d=2c$,
$$\sum_{\cls Y\in\F^{\ell,\ell}_{\sym}}\frac{r^*(I_d,Y)}{o(Y)}\,h_Y=
\sum_{\cls Y\in\F^{\ell,\ell}_{\sym}}\frac{r^*(J_d,Y)}{o(Y)}\,h_Y=
0,$$
and when $d=2c+1$,
\begin{align*}
\sum_{\cls Y\in\F^{\ell,\ell}_{\sym}}\frac{r^*(I_d,Y)}{o(Y)}\,h_Y
&=
-\sum_{\cls Y\in\F^{\ell,\ell}_{\sym}}\frac{r^*(J_d,Y)}{o(Y)}\,h_Y\\
&=
(-1)^k\varepsilon^cN(\calP)^{c-2k}
\bgamma(c,k).
\end{align*}
\end{enumerate}
\end{lem}

\begin{proof}
We have 
\begin{align*}
&\sum_{\cls Y\in\F^{\ell,\ell}_{\sym}}\frac{r^*(I_d,Y)}{o(Y)\,}h_Y
=\frac{r^*(I_d,0_{\ell})}{o(0_{\ell})}\\
&\quad
+\sum_{a=1}^{\ell}\left(\frac{r^*(I_d,I_a\perp 0_{\ell-a})}{o(I_a\perp0_{\ell-a})}
\,h_{I_a\perp 0_{\ell-a}}
+\frac{r^*(I_d,J_a\perp 0_{\ell-a})}{o(J_a\perp0_{\ell-a})}
\,h_{J_a\perp 0_{\ell-a}}\right).
\end{align*}
\begin{align*}
&o(I_{2b+1})N(\calP)^{2bs}\bnu(s,0)=N(\calP)^{2b}\bmu(s,1)o(I_{2b+1}\perp 0_{s-1})\\
&\quad=r^*(I_{2b+1},1)o(I_{2b}\perp0_s)
=r^*(J_{2b+1},1)o(J_{2b}\perp0_s).
\end{align*}

(a)  Suppose that $\ell=2k$.  
Then using Lemma 5.1, when $d=2c$ we get
\begin{align*}
&r^*(I_{2b+1},1)r^*(I_d,I_{2b}\perp 0_{2(k-b)})
+r^*(J_{2b+1},1)r^*(I_d,J_{2b}\perp 0_{2(k-b)})\\
&\quad=
2N(\calP)^{(k-b)(2k-2b-1)+2cb-b^2}\\
&\qquad\cdot
\bmu\bdelta(c-1,2k-b-1)(N(\calP)^c-\varepsilon^c)(N(\calP)^{c-2(k-b)}+\varepsilon^c)
\end{align*}
and
\begin{align*}
&N(\calP)^{2b}(N(\calP)^{2(k-b)}-1)\\
&\qquad\cdot
\left(r^*(I_d,I_{2b+1}\perp 0_{2(k-b)-1})
+r^*(I_d,J_{2b+1}\perp 0_{2(k-b)-1})\right)\\
&\quad=
2N(\calP)^{(k-b)(2k-2b-1)+2cb-b^2+c-2(k-b)}\\
&\qquad\cdot
\bmu\bdelta(c-1,2k-b-1)(N(\calP)^c-\varepsilon^c)(N(\calP)^{2(k-b)}-1).
\end{align*}
So using that $\bmu\bdelta(t,s+s')=\bmu\bdelta(t,s)\bmu\delta(t-s,s'),$ we have
\begin{align*}
\sum_{\cls Y\in\F^{\ell,\ell}_{\sym}}\frac{r^*(I_{2c},Y)}{o(Y)}\,h_Y
&=
\sum_{b=0}^k(-1)^b\frac{N(\calP)^{2b(b+c-2k)}\bmu\bdelta(c,2k-b)\bmu\bdelta(k,b)} {\bmu\bdelta(b,b)\bmu\bdelta(k-b,k-b)\bmu\bdelta(k,b)}\\
&=
\bgamma(k,c)\cdot S(k,c)
\end{align*}
where
$$S(k,c)=\sum_{b=0}^k(-1)^b N(\calP)^{2b(b+c-2k)}
\bmu\bdelta(c-k,k-b)\bgamma(k,b).$$
Since $\bgamma(k,b)=N(\calP)^{2b}\bgamma(k-1,b)+\bgamma(k-1,b-1),$
we find that
\begin{align*}
S(k,c)
&=
-S(k-1,c-1)=(-1)^k S(0,c-k)=(-1)^k,
\end{align*}
proving one case of (a).
We follow this same line of argument when replacing $I_{2c}$ by $J_{2c}$, and when replacing $2c$ by $2c+1$.

(b)  Suppose that $\ell=2k+1$.  
Using the definition of $h_Y$, we have
\begin{align*}
\sum_{\cls Y\in\F^{\ell,\ell}_{\sym}}
\frac{r^*(I_d,Y)}{o(Y)}\, h_Y
&=\sum_{b=0}^k
\left(r^*(I_d,I_{2b+1}\perp0_{2(k-b)}) - r^*(I_d,J_{2b+1}\perp0_{2(k-b)})\right)\\
&\quad\cdot
\frac{h_{I_{2b+1}\perp0_{2(k-b)}}}{o(I_{2b+1}\perp0_{2(k-b)})}.
\end{align*}
When $d$ is even, $r^*(I_d,I_{2b+1}\perp 0_{2(k-b)})
=r^*(I_d,J_{2b+1}\perp0_{2(k-b)}),$
so when $d$ is even the above sum on $\cls Y$ is 0.  So suppose that $d=2c+1$.
Then with $S(k,c)$ as in case (a), we have
\begin{align*}
\sum_{\cls Y\in\F^{\ell,\ell}_{\sym}}
\frac{r^*(I_d,Y)}{o(Y)}\, h_Y
&=\varepsilon^c N(\calP)^{c-2k} \bgamma(c,k)
S(c,k)\\
&=(-1)^k\varepsilon^c N(\calP)^{c-2k} \bgamma(c,k).
\end{align*}

To evaluate the sum on $\cls Y$ when $I_d$ is replaced by $J_d$, we first note that
for any $s\ge0$, when $d$ is even we have 
$r^*(J_d,I_{2b+1}\perp 0_{s})=r^*(J_d,J_{2b+1}\perp0_{s})$, and
when $d$ is odd we have
$r^*(J_d,I_{2b+1}\perp 0_{s})=r^*(I_d,J_{2b+1}\perp0_{s})$ .  
So mimicking our above analysis, we find that when $d$ is even, the sum on $\cls Y =0$, and when $d$ is odd with $d=2c+1$, 
we have
$$\sum_{\cls Y\in\F^{\ell,\ell}_{\sym}}
\frac{r^*(J_d,Y)}{o(Y)}\,h_Y
= -\sum_{\cls Y\in\F^{\ell,\ell}_{\sym}}
\frac{r^*(I_d,Y)}{o(Y)}\,h_Y.$$
\end{proof}


\begin{thebibliography}{0}




\bibitem {Ger} L. Gerstein, \emph{Basic Quadratic Forms}, Graduate
Studies in Math. Vol. 90, Amer. Math. Soc., 2008.


\bibitem{S} H. Saito, ``A generalization of Gauss sums and its applications to Siegel modular forms and $L$-functions associated with the vector space of quadratic forms."
\textit{J. reine angew. Math} (1991), 91-142.


 \bibitem{half-int-aps} L.H. Walling, ``A formula for the action of Hecke operators on half-integral weight Siegel modular forms and applications."  \textit{J. Number Theory} (2013), no. 5, 1608-1644.

\end{thebibliography}
\end{document}